\documentclass[12pt,a4paper,twoside]{article}

\usepackage{amsmath,amssymb,amsthm,amsfonts}
\usepackage{mathtools}
\usepackage{latexsym}
\usepackage{mathrsfs}
\usepackage{extarrows} 
\usepackage{enumerate} 
\usepackage{enumitem} 
\usepackage{bm} 
\usepackage{arydshln}

\usepackage[dvipdfmx,hypertexnames=false,setpagesize=false]{hyperref}
\usepackage{cleveref}
\usepackage{cite}

\usepackage{graphicx,color} 
\usepackage[svgnames]{xcolor} 
\usepackage[normalem]{ulem} 
\usepackage{color}

\usepackage{mathrsfs} 
\usepackage[dvips,all]{xy}  
\usepackage{array, booktabs} 
\usepackage{colortbl} 
\usepackage{tabularx} 
\usepackage{float} 
\usepackage{subcaption} 
\usepackage{longtable}
\usepackage[margin=25truemm]{geometry}

\makeatletter
\renewcommand{\theenumi}{{\upshape (\@roman\c@enumi)}}

\renewcommand{\p@enumii}{}
\renewcommand{\theenumii}{{\upshape (\@alph\c@enumii)}}

\makeatother

\crefname{equation}{式}{式}
\crefname{figure}{図}{図}
\crefname{table}{表}{表}
\crefname{algorithm}{Algorithm}{Algorithm}

\theoremstyle{definition}
\newtheorem{theo}{Theorem}[section]
\newtheorem{dfn}[theo]{Definition}
\newtheorem{prop}[theo]{Proposition}
\newtheorem{lem}[theo]{Lemma}
\newtheorem{cor}[theo]{Corollary}
\newtheorem{rem}[theo]{Remark}
\newtheorem{ex}{Example}

\title{Multiple modular $L$-functions and modular iterated integrals}
\author{Mahiro Yokomizo}
\date{\today}

\begin{document}
\maketitle
\begin{abstract}
The connection between multiple modular $L$-functions, as defined by Manin in \cite{M}, and modular iterated integrals was made explicit by Choie and Ihara \cite{i} under the restrictive assumption that all modular forms involved have vanishing constant terms in their $q$-expansions. In this paper, we remove the assumption and establish the relationship between modular iterated integrals and multiple modular $L$-functions for general modular forms, including those with nonzero constant terms. We also provide a proof of a functional equation for modular iterated integrals, which is a specialization of a general result obtained by Brown \cite{IB}. This leads us to a generalization of the result of Choie--Ihara\cite{i}. In the final part of the paper, we compute explicit examples of modular iterated integrals. These calculations essentially reproduce the explicit initial computations carried out by Brown\cite{IB}, but they also serve to validate the broader framework developed in this work.
\end{abstract}

\section{Introduction}

Modular forms and their associated $L$-functions play an important role in number theory and related fields. These $L$-functions possess rich arithmetic properties and are deeply connected to periods, special values, and functional equations. In an effort to generalize classical modular $L$-functions, Manin \cite{M} introduced the concept of \emph{modular iterated integrals} and defined associated \emph{multiple modular $L$-functions}, extending ideas related to multiple zeta values.

Choie and Ihara \cite{i} established an explicit correspondence between multiple modular $L$-functions and modular iterated integrals, based on Manin's definitions. However, their results relied on the assumption that all modular forms involved have vanishing constant terms in their $q$-expansions—that is, they are cusp forms. This restriction excludes many important modular forms such as Eisenstein series, thereby limiting the generality of their framework.

The aim of this paper is to remove the assumption and extend the theory to general modular forms, including those with nonzero constant terms. To achieve this, we develop methods to systematically isolate and manage the contribution of constant terms. This allows us to generalize the explicit relationships given by Choie--Ihara, and also to revisit and prove functional equations in the context of modular iterated integrals, inspired by earlier results of Brown \cite{IB}.

Our main results (\Cref{thI}, \Cref{thS}) can be roughly stated as follows:

\begin{theo}[Main Theorem (Rough statement)]
For positive integers $n$ , let $f_1, \ldots, f_n$ be modular forms of level $N$, not necessarily cusp forms. Then the modular iterated integral of the form
 $I_{i\infty}^0\binom{f_1,\dots,f_n}{s,\alpha_2,\dots,\alpha_n}$ 
with $s \in \mathbb{C}$ and $\alpha_2, \ldots, \alpha_n \in \mathbb{Z}_{>0}$, can be expressed explicitly as a linear combination of multiple modular $L$-functions,
and conversely, each such multiple modular $L$-function  $L\binom{f_1, \ldots, f_n;}{ s, \alpha_2, \ldots, \alpha_n}$  can be written as a linear combination of modular iterated integrals of this form.
\end{theo}

In the final section of this paper, we compute explicit examples of modular iterated integrals on modular curves to demonstrate the effectiveness and consistency of the general framework. These explicit examples extend and confirm prior computations carried out by Brown \cite{IB}.

This work clarifies and expands the relationship between modular forms, iterated integrals, and multiple $L$-functions, contributing to a deeper understanding of their analytic structure and potential arithmetic applications.

\section*{Acknowledgements}

The author would like to express his sincere gratitude to Prof.~Yasuo Ohno for his insightful suggestions and constant encouragement.
The author is also grateful to Prof.\ Takuya Yamauchi for his helpful comments and valuable advice, particularly for his insightful guidance on modular forms for $\Gamma_0(4)$.

\section{Iterated Integrals of 1-Forms}
In this section, we define iterated integrals and describe some of their basic properties.

\begin{dfn}
Let $M$ be a differentiable manifold, let $\eta_1, \dots, \eta_k$ be $1$-forms on $M$, and let $\gamma: [0,1] \to M$ be a piecewise smooth path. Then the multiple integral
\begin{align*}
\int_\gamma \eta_1 \dots \eta_k 
&:= \int_{0 < t_1 < \dots < t_k < 1} \gamma^* \eta_1(t_1) \dots \gamma^* \eta_k(t_k) \\
&= \int_0^1 \gamma^* \eta_k(t_k) \int_0^{t_k} \gamma^* \eta_{k-1}(t_{k-1}) \dots \int_0^{t_2} \gamma^* \eta_1(t_1)
\end{align*}
is called the \emph{iterated integral} of the $1$-forms $\eta_1, \dots, \eta_k$ along $\gamma$. Here, $\gamma^* \eta_j$ denotes the pullback of $\eta_j$ by $\gamma$.
\end{dfn}

\begin{ex}\label{mzv}
For any integers $k_1, \dots, k_{d-1} \in \mathbb{Z}_{>0}$ and $k_d \in \mathbb{Z}_{>1}$, we have:
\begin{align*}
\zeta(k_1, \dots, k_d)
= \int_{\mathrm{dch}} \omega_1 \omega_0^{k_1 - 1} \dots \omega_1 \omega_0^{k_{d-1}},
\end{align*}
where the path $\mathrm{dch}: [0,1] \to \mathbb{C}$ is defined by $\mathrm{dch}(t) := t$, and the $1$-forms are given by $\omega_0 := \frac{dz}{z}$ and $\omega_1 := \frac{dz}{1 - z}$. Here, $M = \mathbb{C}$ and $z$ is a complex variable with $dz = dx + i\,dy$.

The quantity $\zeta(k_1, \dots, k_d)$ is the \emph{multiple zeta value}, a real number defined by
\begin{align*}
\zeta(k_1, \dots, k_d)
:= \sum_{0 < n_1 < \dots < n_d} \frac{1}{n_1^{k_1} \dots n_d^{k_d}}.
\end{align*}
\end{ex}

Iterated integrals satisfy the following fundamental properties.

\begin{prop}[]\label{int}
For positive integer $k$ and $l$, let $\eta_1, \dots, \eta_k, \eta_{k+1}, \dots, \eta_{k+l}$ be regular differential $1$-forms on a differentiable manifold $M$, and let $\gamma, \gamma_1, \gamma_2: [0,1] \to M$ be piecewise smooth paths, chosen so that all the integrals below converge. Then the following properties hold:
\begin{enumerate}
\item The value $\int_\gamma \eta_1 \dots \eta_k$ is independent of the parametrization of $\gamma$.
\item If $\eta_1, \dots, \eta_k$ are closed forms and satisfy $\eta_i \wedge \eta_{i+1} = 0$ for $i = 1, \dots, k-1$, then $\int_\gamma \eta_1 \dots \eta_k$ depends only on the homotopy class of $\gamma$.
\item For $1$-forms $\eta_1, \dots, \eta_k, \eta_{k+1}, \dots, \eta_{k+l}$ and a path $\gamma$, we have
\begin{align*}
\int_\gamma \eta_1 \dots \eta_k \cdot \int_\gamma \eta_{k+1} \dots \eta_{k+l}
= \sum_{\sigma \in S_{k,l}} \int_\gamma \eta_{\sigma(1)} \dots \eta_{\sigma(k+l)},
\end{align*}
where $S_{k,l}$ denotes the set of $(k,l)$-shuffles.
\item For $1$-forms $\eta_1, \dots, \eta_k$ and a path $\gamma$, we have
\begin{align*}
\int_{\gamma^{-1}} \eta_1 \dots \eta_k = (-1)^k \int_\gamma \eta_k \dots \eta_1.
\end{align*}
\item For $1$-forms $\eta_1, \dots, \eta_k$ and paths $\gamma_1, \gamma_2$, we have
\begin{align*}
\int_{\gamma_1 \gamma_2} \eta_1 \dots \eta_k
= \sum_{j=0}^k \int_{\gamma_1} \eta_1 \dots \eta_j \cdot \int_{\gamma_2} \eta_{j+1} \dots \eta_k.
\end{align*}
\end{enumerate}
\end{prop}

\section{Modular Iterated Integrals}

Let $\mathbb{H} = \{ z \in \mathbb{C} \mid \mathrm{Im}\, z > 0 \}$ be the upper-half plane with its action of $\mathrm{SL}_2(\mathbb{Z})$:
\begin{align*}
	\gamma z = \frac{az + b}{cz + d}, \quad \text{where } \gamma =
	\begin{pmatrix}
		a & b \\
		c & d
	\end{pmatrix}
	\in \mathrm{SL}_2(\mathbb{Z}).
\end{align*}
Let $\Gamma$ be a congruence subgroup of $\mathrm{SL}_2(\mathbb{Z})$ and let $k$ be an integer.

For any $\gamma \in \mathrm{SL}_2(\mathbb{Z})$, and a function $f: \mathfrak{H} \to \mathbb{C}$, we define the slash operator by
\begin{align*}
	f|[\gamma]_k(z) := (cz + d)^{-k} f(\gamma z).
\end{align*}
We say that a holomorphic function
\begin{align*}
	f: \mathfrak{H} \to \mathbb{C}
\end{align*}
is a modular form of weight $k$ with respect to $\Gamma$ if it satisfies
\begin{align*}
	f|[\gamma]_k(z) = f(z) \quad \text{for any } \gamma \in \Gamma,
\end{align*}
and $f$ is holomorphic at all cusps, i.e., at $\mathbb{Q} \cup \{ \infty \}$.\\
In this paper, we only consider the congruence subgroup $\Gamma_0(N)$.
We now define the modular iterated integrals.
\begin{dfn}
Let $U \subset \mathbb{C}$ be a simply connected domain, and let $a, b \in U$. Suppose $f_1, \ldots, f_n$ are holomorphic functions on $U$, and let $s_1, \ldots, s_n \in \mathbb{C}$. Then we define the iterated integral
\[
I_a^b\binom{f_1, \ldots, f_n}{s_1, \ldots, s_n}
:= \int_a^b f_1(z_1) z_1^{s_1} \frac{dz_1}{z_1}
\int_a^{z_1} f_2(z_2) z_2^{s_2} \frac{dz_2}{z_2}
\cdots
\int_a^{z_{n-1}} f_n(z_n) z_n^{s_n} \frac{dz_n}{z_n}.
\]
Moreover, when $U = \mathbb{H}$ and each $f_i \in M_{k_i}(\Gamma_0(N), \chi_i)$, we call
\[
I_{i\infty}^0\binom{f_1, \ldots, f_n}{s_1, \ldots, s_n}=\int_{i\infty}^0 f_1(z_1) z_1^{s_1} \frac{dz_1}{z_1}\int_{i\infty}^{z_1} f_2(z_2) z_2^{s_2} \frac{dz_2}{z_2}\cdots\int_{i\infty}^{z_{n-1}} f_n(z_n) z_n^{s_n} \frac{dz_n}{z_n}
\]
a modular iterated integral.
\end{dfn}
In the following, we show that the above definition is well-defined, that is, it admits an analytic continuation to a meromorphic function on $\mathbb{C}^n$, and moreover satisfies a functional equation.

\begin{lem}\label{lem1}
Let $b \in \mathbb{C}$, and suppose
\[
\mathrm{Re}(s_n) < 0, \quad \mathrm{Re}(s_n + s_{n-1}) < 0, \quad \ldots, \quad \mathrm{Re}(s_n + \cdots + s_1) < 0.
\]
Then the following holds:
\begin{align}\label{a}
I_{i\infty}^{b} \binom{1, \ldots, 1}{s_1, \ldots, s_n}
= \frac{b^{s_1 + \cdots + s_n}}{s_n(s_n + s_{n-1}) \cdots (s_n + \cdots + s_1)}.
\end{align}
\end{lem}

\begin{proof}
We prove the claim by induction on $n$. The case $n = 1$ is clear.

Assume that the statement holds for $n-1$. Then we compute
\begin{align*}
I_{i\infty}^{b} \binom{1, \ldots, 1}{s_1, \ldots, s_n}
&= \int_{i\infty}^{b} z_1^{s_1} \frac{dz_1}{z_1} 
\int_{i\infty}^{z_1} z_2^{s_2} \frac{dz_2}{z_2} \cdots 
\int_{i\infty}^{z_{n-1}} z_n^{s_n} \frac{dz_n}{z_n} \\
&= \int_{i\infty}^{b} \frac{z_1^{s_1 + \cdots + s_n}}{s_n(s_n + s_{n-1}) \cdots (s_n + \cdots + s_2)} \frac{dz_1}{z_1} \\
&= \frac{b^{s_1 + \cdots + s_n}}{s_n(s_n + s_{n-1}) \cdots (s_n + \cdots + s_1)}.
\end{align*}
This completes the proof.
\end{proof}

This lemma shows that the left-hand side of (\ref{a}) converges under the condition
\[
\mathrm{Re}(s_n) < 0, \quad \mathrm{Re}(s_n + s_{n-1}) < 0, \quad \ldots, \quad \mathrm{Re}(s_n + \cdots + s_1) < 0,
\]
while the right-hand side defines a rational function on $\mathbb{C}^n$, meromorphic with possible poles only along the hyperplanes
\[
s_n = 0, \quad s_n + s_{n-1} = 0, \quad \ldots, \quad s_n + \cdots + s_1 = 0.
\]
Thus, the \cref{lem1} gives the analytic continuation of the iterated integral $I_{i\infty}^{b} \binom{1, \ldots, 1}{s_1, \ldots, s_n}$ to a meromorphic function on $\mathbb{C}^n$.

\begin{lem}\label{lem2}
Let $(\{1\}^n) = (1,\dots,1)$ denote the sequence consisting of $n$ copies of $1$, and suppose $n = n_1 + \dots + n_{l+1} + l$. Then
\begin{align*}
I_{i\infty}^{\frac{i}{\sqrt{N}}}\binom{\{1\}^{n_{1}},f_{n_1'}^{0},\dots,\{1\}^{n_{l}},f_{n_l'}^{0},\{1\}^{n_{l+1}}}{s_{1},\dots,s_{n}}
\end{align*}
extends to a meromorphic function on $\mathbb{C}^n$ with simple poles along the divisors
\[
s_n = 0, \quad s_n + s_{n-1} = 0, \quad \dots, \quad s_n + \dots + s_{n-n_{l+1}} = 0.
\]
Here, $n_l' = n_1 + \dots + n_l + l$.
\end{lem}

\begin{proof}
If $n_{l+1} = 0$, the function is holomorphic, since $f_i^0(z)$ decays rapidly as $z \to i\infty$. When $n_{l+1} \neq 0$, assume that
\[
\mathrm{Re}(s_n) < 0,\quad \mathrm{Re}(s_ + \dots + s_{n-n_{l+1}}) < 0.
\]
Then, by \cref{lem1}, we have:
\begin{align*}
&I_{i\infty}^{\frac{i}{\sqrt{N}}}\binom{\{1\}^{n_{1}},f_{n_1'}^{0},\dots,\{1\}^{n_{l}},f_{n_l'}^{0},\{1\}^{n_{l+1}}}{s_{1},\dots,s_{n}} \\
&= \frac{1}{s_n(s_n + s_{n-1})\dots(s_n + \dots + s_{n-n_{l+1}})}
I_{i\infty}^{\frac{i}{\sqrt{N}}}\binom{\{1\}^{n_{1}},f_{n_1'}^{0},\dots,\{1\}^{n_{l}},f_{n_l'}^{0}}{s_{1},\dots,s_{n_l'-1},s_{n_l'} + \dots + s_n}.
\end{align*}
This expression shows that the function admits a meromorphic continuation.
\end{proof}

\begin{theo}[\cite{IB}]
The modular iterated integral admits a meromorphic continuation to $\mathbb{C}^{n}$, with poles located along the divisors
\begin{align*}
&s_{l}=0,\quad s_{l}+s_{l-1}=0,\quad \dots,\quad s_{l}+\dots+s_{1}=0,\\
&s_{l+1}=k_{l+1},\quad s_{l+1}+s_{l+2}=k_{l+1}+k_{l+2},\quad \dots,\quad s_{l+1}+\dots+s_{n}=k_{l+1}+\dots+k_{n}
\end{align*}
for $l=1,\dots,n$.
\end{theo}

\begin{proof}
By \cref{int}, we have:
\begin{align*}
I_{i\infty}^{0}\binom{f_{1},\dots,f_{n}}{s_{1},\dots,s_{n}}
= \sum_{j=0}^{n} I_{i\infty}^{\frac{i}{\sqrt{N}}}\binom{f_1,\dots,f_j}{s_1,\dots,s_j} \cdot I_{\frac{i}{\sqrt{N}}}^{0}\binom{f_{j+1},\dots,f_n}{s_{j+1},\dots,s_n}.
\end{align*}
Consider the change of variable $z \mapsto -\frac{1}{\sqrt{N}z}$ in the iterated integral $
I_{\frac{i}{\sqrt{N}}}^{0}\binom{f_{j+1},\dots,f_n}{s_{j+1},\dots,s_n}$.
Then we obtain:
\begin{align*}
&I_{\frac{i}{\sqrt{N}}}^{0}\binom{f_{j+1},\dots,f_n}{s_{j+1},\dots,s_n} \\
&= (-1)^{s_{j+1}+\dots+s_n} N^{-s_{j+1}-\dots-s_n + \frac{k_{j+1}+\dots+k_n}{2}}
I_{i\infty}^{\frac{i}{\sqrt{N}}}\binom{\tilde{f}_n,\dots,\tilde{f}_{j+1}}{k_n - s_n, \dots, k_{j+1} - s_{j+1}},
\end{align*}
where $\tilde{f} := f|[\omega_N]_k$ denotes the slash operator of weight $k$.

Now writing $f = f^0 + a_0$ and applying \cref{lem2}, we see that $I_{i\infty}^{0}\binom{f_{1},\dots,f_{n}}{s_{1},\dots,s_{n}}$
admits a meromorphic continuation to $\mathbb{C}^n$.
\end{proof}

This calculation yields the following functional equation:

\begin{cor}[Functional Equation]
Define
\[
Z\binom{f_1,\dots,f_n}{s_1,\dots,s_n} := N^{\frac{s_1 + \dots + s_n}{2}} I_{i\infty}^{0} \binom{f_1,\dots,f_n}{s_1,\dots,s_n}.
\]
Then we have the identity:
\[
Z\binom{f_1,\dots,f_n}{s_1,\dots,s_n}
= (-1)^{s_1+\dots+s_n} Z\binom{\tilde{f}_n,\dots,\tilde{f}_1}{k_n - s_n, \dots, k_1 - s_1}.
\]
\end{cor}

\begin{rem}
Let $f$ satisfy the functional equation $f(t^{-1}) = \varepsilon t^w f(t)$, and admit a decomposition $f = f^{0} + f^{\infty}$, where $f^{\infty} \in \mathbb{C}[t]$ and $f^{0}$ decays rapidly to zero as $t \to \infty$. 
In this setting, Brown\cite{IB} proved the same result.
\end{rem}

\section{Multiple Modular $L$-Functions and Modular Iterated Integrals}

In this section, we define multiple modular $L$-functions and explain their connection with modular iterated integrals.

\begin{dfn}
Let $N$ be a positive integer, and let $f_i = \sum_{j=0}^{\infty} a_j^{(i)} q^j$ be a modular form in $M_{k_i}(\Gamma_0(N), \chi_i)$. Then the multiple modular $L$-function is defined by
\begin{align*}
L\binom{f_1, \dots, f_n}{s_1, \dots, s_n}
:= (-2\pi i)^{-(s_1 + \dots + s_n)}
\sum_{m_1, \dots, m_n > 0}
\frac{a_{m_1}^{(1)} \cdots a_{m_n}^{(n)}}
{(m_1 + \dots + m_n)^{s_1}
(m_2 + \dots + m_n)^{s_2}
\cdots m_n^{s_n}}.
\end{align*}
\end{dfn}

\begin{rem}
The series defining the multiple modular $L$-function converges when the real part of $s_1$ is sufficiently large.
\end{rem}

\begin{ex}
Let $n = 2$, and let
\[
f_i = -\frac{B_{k_i}}{2k_i} + \sum_{j=1}^\infty \sigma_{k_i - 1}(j) q^j,
\]
where $k_i$ are even integers. Then each $f_i$ is a modular form. The multiple modular $L$-function associated to $f_1$ and $f_2$ is given by
\begin{align*}
(-2\pi i)^{s_1 + s_2} L\binom{f_1, f_2}{s_1, s_2}
&= \sum_{n_1, n_2 > 0}
\frac{\sigma_{k_1 - 1}(n_1) \sigma_{k_2 - 1}(n_2)}
{(n_1 + n_2)^{s_1} n_2^{s_2}} \\
&= \sum_{n_1, m_1, n_2, m_2 > 0}
\frac{m_1^{k_1 - 1} m_2^{k_2 - 1}}
{(n_1 m_1 + n_2 m_2)^{s_1} (n_2 m_2)^{s_2}}.
\end{align*}
\end{ex}

\begin{rem}
This example shows that multiple modular $L$-functions of Eisenstein series can be expressed in terms of the $L$-functions of multiple Eisenstein series.
\end{rem}

In what follows, we present a formula that expresses $I_{i\infty}^{0}\binom{f_1, \dots, f_n}{s_1, \dots, s_n}$
in terms of $L\binom{f_1, \dots, f_n}{s_1, \dots, s_n},$
as well as a theorem stating the converse expression.

\begin{theo}\label{thI}
Let $N$ be a positive integer, and let $f_j \in M_{k_j}(\Gamma_0(N), \chi_j)$ for $j = 1, \dots, n$, with $s \in \mathbb{C}$ and $\alpha_2, \dots, \alpha_n \in \mathbb{Z}_{>0}$. Then the following holds:
\begin{align*}
&I_{i\infty}^{0} \binom{f_1, \dots, f_n}{s, \alpha_2, \dots, \alpha_n} \notag \\
&= \sum_{l=1}^{n} \sum_{\substack{n_{l+1}' = n+1 \\ n_1, \dots, n_{l+1} \ge 0}} \Gamma^{(s, \alpha_2, \dots, \alpha_{n_l'-1}, \alpha_{n_l',n})} A_{n_1, \dots, n_l} B_{n_{l+1}}(s, \alpha_2, \dots, \alpha_n) \notag \\
&\quad \times \sum_{\substack{0 \le j_r < \alpha_r + j_{r+1} \\ (2 \le r \le n_l') \\ j_{n_l'+1}':=0}} \binom{s + j_2 - 1}{j_2}  \prod_{k=3}^{n_l'} \binom{\alpha_{k-1} + j_k - 1}{j_k} \notag \\
&\quad \times L\binom{f_{n_1'}, \dots, f_{n_l'}}{s + \alpha_{2,n_1'} + j_{n_1'+1}, \boldsymbol{\alpha}_{n_1,\dots,n_l}(\mathbb{J}), \alpha_{n_{l-1}'+1,n} - j_{n_{l-1}'+1}},
\end{align*}
where
\begin{align*}
\boldsymbol{\alpha}_{n_1,\dots,n_l}(\mathbb{J}) 
&= (\alpha_{n_1'+1, n_2'} - j_{n_1'+1} + j_{n_2'+1}, \dots, \alpha_{n_{l-2}'+1, n_{l-1}'} - j_{n_{l-2}'+1} + j_{n_{l-1}'+1}), \\
n_j' &= n_1 + \dots + n_j + j, \\
\alpha_{n,m} &= \alpha_n + \alpha_{n+1} + \dots + \alpha_m, \\
\Gamma^{(s_1, \dots, s_n)} &= (-1)^n \Gamma(s_1) \dots \Gamma(s_n), \\
A_{n_1,\dots,n_l} &= \frac{\prod_{i=1}^n a_0^{(i)}}{a_0^{n_1'} \cdots a_0^{n_l'}}, \\
B_m(s_1, \dots, s_n) &=\prod_{k=1}^{m}\frac{1}{s_n+\dots+s_{n-k-1}}
\end{align*}
\end{theo}
\begin{cor}
Let $N$ be a positive integer, and let $f_j \in S_{k_j}(\Gamma_0(N), \chi_j)$ for $j = 1, \dots, n$, with $s \in \mathbb{C}$ and $\alpha_2, \dots, \alpha_n \in \mathbb{Z}_{>0}$. Then the following holds:
\begin{align*}
&I_{i\infty}^{0} \binom{f_1, \dots, f_n}{s, \alpha_2, \dots, \alpha_n} \notag \\
&= \Gamma^{(s,\alpha_2,\dots,\alpha_n)}\sum_{\substack{0 \le j_r < \alpha_r + j_{r+1} \\ (2 \le r \le n) \\ j_{n+1}':=0}}\binom{s + j_2 - 1}{j_2}  \prod_{k=3}^{n_l'} \binom{\alpha_{k-1} + j_k - 1}{j_k}\\
&\times L\binom{f_1,\dots,f_n}{s+j_2,\alpha_2-j_2+j_3,\dots,\alpha_n-j_n+j_{n+1}}
\end{align*}
\end{cor}
\begin{proof}
Because \( f_i \) is a cusp form, it follows that \( A_{n_1,\dots,n_l} = 0 \) for all
\( l \neq n \). Therefore, the sum runs only over the case
\( l = n \). In this case, we necessarily have \( n_i = 0 \) for all
\( i \), which completes the proof.
\end{proof}
\begin{theo}\label{thS}
Let $N$ be a positive integer, and let $f_j \in M_{k_j}(\Gamma_0(N), \chi_j)$ for $j = 1, \dots, n$, with $s \in \mathbb{C}$ and $\alpha_2, \dots, \alpha_n \in \mathbb{Z}_{>0}$. Then $L\binom{f_1, \dots, f_n}{s, \alpha_2, \dots, \alpha_n}$
can be expressed as a linear combination of iterated modular integrals.
\end{theo}

\begin{cor}
The function $L\binom{f_1, \dots, f_n}{s, \alpha_2, \dots, \alpha_n}$
admits a meromorphic continuation to $\mathbb{C}$.
\end{cor}

\begin{ex}
In Theorem~\ref{thI}, when $n = 2$ and $\alpha_2 = 2$, we obtain:
\begin{align*}
I_{i\infty}^{0} \binom{f_1, f_2}{s, 2}
&= \Gamma(s+1) L\binom{f_1, f_2}{s+1, 1} + \Gamma(s) L\binom{f_1, f_2}{s, 2} - \frac{a_0^{(2)} \Gamma(s+2)}{2} L\binom{f_1}{s+2} \notag \\
&\quad + a_0^{(1)} \Gamma(s+1) L\binom{f_2}{s+2} + a_0^{(1)} \Gamma(s) L\binom{f_2}{s+2}.
\end{align*}
\end{ex}

\begin{ex}
In Theorem~\ref{thS}, the expression becomes:
\begin{align*}
\Gamma^{(s, 2)} L\binom{f_1, f_2}{s, 2}
&= I_{i\infty}^{0} \binom{f_1, f_2}{s, 2} - I_{i\infty}^{0} \binom{f_1, f_2}{s+1, 1}
+ \frac{a_0^{(2)}}{2} I_{i\infty}^{0} \binom{f_1}{s+2}
+ \frac{a_0^{(1)}}{s(s+1)} I_{i\infty}^{0} \binom{f_2}{s+2}.
\end{align*}
\end{ex}

Below, we write
\[
F_a^z\binom{f_1, f_2, \dots, f_n}{-, s_2, \dots, s_n}
:= f_1(z) I_a^z\binom{f_2, \dots, f_n}{s_2, \dots, s_n}.
\]
Although the proof of the theorem is based on integration by parts, the above definition is not convenient for direct computation. In order to handle constant terms more effectively, we introduce an alternative, computationally favorable formulation.

\begin{dfn}
Let $\alpha_1, \dots, \alpha_n$ be integers and $a \in \mathfrak{H}$. We define the holomorphic functions $\tilde{I}_a^z$ and $\tilde{F}_a^z$ by
\begin{align*}
\tilde{I}_a^z \binom{f_1, \dots, f_n}{\alpha_1, \dots, \alpha_n}
&:= \int_a^z f_1(z_1) (z_1 - z)^{\alpha_1} \frac{dz_1}{z_1 - z}
\int_a^{z_1} \cdots \int_a^{z_{n-1}} f_n(z_n) (z_n - z_{n-1})^{\alpha_n} \frac{dz_n}{z_n - z_{n-1}}, \\
\tilde{F}_a^z \binom{f_1, \dots, f_n}{-, \alpha_2, \dots, \alpha_n}
&:= f_1(z) \tilde{I}_a^z \binom{f_2, \dots, f_n}{\alpha_2, \dots, \alpha_n}.
\end{align*}
\end{dfn}

This definition is suitable for performing integration by parts. However, due to the presence of constant terms, direct computation is not straightforward. Therefore, we must first extract the contribution of the constant terms.

\begin{prop}\label{prop1}
Let $f_i \in M_{k_i}(\Gamma_0(N), \chi_i)$. Then the following identity holds:
\begin{align*}
&I_{i\infty}^{z} \binom{f_1, \dots, f_n}{s_1, \dots, s_n} \notag \\
&= \sum_{l=1}^{n} \sum_{\substack{n_{l+1}' = n+1 \\ n_1, \dots, n_{l+1} \ge 0}} 
A_{n_1, \dots, n_l} \, B_{n_{l+1}}(s_1, \dots, s_n) \,
I_{i\infty}^{z} \binom{\{1\}^{n_1}, f_{n_1'}^0, \{1\}^{n_2}, f_{n_2'}^0, \dots, \{1\}^{n_l}, f_{n_l'}^0}
{s_1, \dots, s_{n_l'-1}, s_{n_l'} + \dots + s_n} \notag \\
&\quad + a_0^{(1)} a_0^{(2)} \cdots a_0^{(n)} \,
I_{i\infty}^z \binom{\{1\}^n}{s_1, \dots, s_n}.
\end{align*}
Here, we define:
\begin{align*}
A_{n_1, \dots, n_l} &= \frac{\prod_{i=1}^n a_0^{(i)}}{a_0^{(n_1')} a_0^{(n_2')} \cdots a_0^{(n_l')}}, 
\quad n_j' := n_1 + \dots + n_j + j, \\
B_m(s_1, \dots, s_n) &=
\begin{cases}
1 & \text{if } m = 0, \\
\frac{1}{s_n (s_n + s_{n-1}) \cdots (s_n + \dots + s_{n - m + 1})} & \text{if } m \neq 0.
\end{cases}
\end{align*}
\end{prop}

\begin{proof}
Writing each $f_i$ as $f_i = f_i^0 + a_0^{(i)}$ and applying multilinearity, we decompose $I_{i\infty}^z$ as follows:
\begin{align*}
&\sum_{l=1}^{n} \sum_{\substack{n_1 + \dots + n_{l+1} = n - l \\ n_1, \dots, n_{l+1} \ge 0}} 
A_{n_1, \dots, n_{l}} \,
I_{i\infty}^z \binom{\{1\}^{n_1}, f_{n_1'}^0, \dots, \{1\}^{n_l}, f_{n_l'}^0, \{1\}^{n_{l+1}}}{s_1, \dots, s_n} \notag \\
&\quad + a_0^{(1)} a_0^{(2)} \cdots a_0^{(n)} 
I_{i\infty}^z \binom{\{1\}^n}{s_1, \dots, s_n}.
\end{align*}
The final integral involving $\{1\}^{n_{l+1}}$ can be computed using \cref{lem1}, which yields the desired formula.
\end{proof}

This proposition allows us to successfully isolate the contributions from constant terms. From here, it suffices to rewrite $F$ in terms of $\tilde{F}$.

\begin{prop}
Let $f_i$ $(i = 1, \dots, n,\ n > 1)$ be holomorphic functions on the upper half-plane $\mathbb{H}$. Then the following holds:
\begin{align*}
(1)\quad
&F_{a}^{z}\binom{f_1,\dots,f_n}{-,\alpha_2,\dots,\alpha_n} \\
&= \sum_{\substack{0 \le j_r < \alpha_r + j_{r+1} \\ (2 \le r \le n) \\ j_{n+1} := 0}} \prod_{k=2}^{n} \binom{\alpha_k + j_{k+1} - 1}{j_k} z^{j_2} 
\tilde{F}_{a}^{z}\binom{f_1,\dots,f_n}{-,\alpha_2 - j_2 + j_3, \dots, \alpha_n - j_n + j_{n+1}}, \\
(2)\quad
&\tilde{F}_{a}^{z}\binom{f_1,\dots,f_n}{-,\alpha_2,\dots,\alpha_n} \\
&= \sum_{\substack{0 \le j_r < \alpha_r \\ (2 \le r \le n) \\ j_{n+1} := 0}} \prod_{k=2}^{n} (-1)^{j_k} \binom{\alpha_k - 1}{j_k} z^{j_2}
F_{a}^{z}\binom{f_1,\dots,f_n}{-,\alpha_2 - j_2 + j_3, \dots, \alpha_n - j_n + j_{n+1}}.
\end{align*}
\end{prop}

\begin{proof}
We prove (1) by induction on $n$. When $n=2$,
\begin{align*}
F_{a}^{z}\binom{f_1, f_2}{-, \alpha_2}
&= f_1(z) \int_{a}^{z} f_2(z_1) z_1^{\alpha_2 - 1} \, dz_1 \\
&= f_1(z) \int_{a}^{z} f_2(z_1) \left( \sum_{0 \le j_2 < \alpha_2} \binom{\alpha_2 - 1}{j_2} z^{j_2} (z_1 - z)^{\alpha_2 - j_2 - 1} \right) dz_1 \\
&= \sum_{0 \le j_2 < \alpha_2} \binom{\alpha_2 - 1}{j_2} z^{j_2} 
\tilde{F}_{a}^{z}\binom{f_1, f_2}{-, \alpha_2 - j_2}.
\end{align*}
Now assume the claim holds for $n - 1$, and consider the case $n > 2$:
\begin{align*}
F_{a}^{z} \binom{f_1,\dots,f_n}{-,\alpha_2,\dots,\alpha_n}
&= f_1(z) \int_{a}^{z} z_2^{\alpha_2 - 1} 
F_{a}^{z_2} \binom{f_2,\dots,f_n}{-,\alpha_3,\dots,\alpha_n} \, dz_2 \notag \\
&= f_1(z) \int_{a}^{z} z_2^{\alpha_2 - 1}
\sum_{\substack{
  0 \le j_r < \alpha_r + j_{r+1} \\
  (3 \le r \le n),\; j_{n+1} := 0
}} 
\prod_{k=3}^{n} \binom{\alpha_k + j_{k+1} - 1}{j_k}
z_2^{j_3} \notag \\
&\quad \times 
\tilde{F}_{a}^{z_2} \binom{f_2,\dots,f_n}
{-, \alpha_3 - j_3 + j_4, \dots, \alpha_n - j_n + j_{n+1}} \, dz_2.
\end{align*}
This proves (1) by induction. Statement (2) follows directly from applying the binomial expansion.
\end{proof}

\begin{cor}\label{col2}
When $n_l' > 1$, the following holds:
\begin{align*}
&F_{a}^{z}\binom{\{1\}^{n_1}, f_{n_1'}^{0}, \{1\}^{n_2}, f_{n_2'}^{0}, \dots, \{1\}^{n_l}, f_{n_l'}^{0}}{-, \alpha_2, \dots, \alpha_{n_l' ,n}} \notag \\
&= \sum_{\substack{0 \le j_r < \alpha_r + j_{r+1} \\ (2 \le r \le n_l')}} \binom{\alpha_{n_l',n-1}}{j_{n_l'}}\prod_{k=2}^{n_l'-1} \binom{\alpha_k + j_{k+1} - 1}{j_k} z^{j_2} \tilde{F}_{a}^{z}
\binom{\{1\}^{n_1}, f_{n_1'}^{0}, \{1\}^{n_2}, f_{n_2'}^{0}, \dots, \{1\}^{n_l}, f_{n_l'}^{0}}{-, \alpha_2 - j_2 + j_3, \dots, \alpha_{n_l'} - j_{n_l'} + j_{n_l'+1}}, \notag \\
&\tilde{F}_{a}^{z}\binom{\{1\}^{n_1}, f_{n_1'}^{0}, \{1\}^{n_2}, f_{n_2'}^{0}, \dots, \{1\}^{n_l}, f_{n_l'}^{0}}{-, \alpha_2, \dots, \alpha_{n_l'}} \notag \\
&= \sum_{\substack{0 \le j_r < \alpha_r \\ (2 \le r \le n_l')}} \prod_{k=2}^{n_l'} (-1)^{j_k} \binom{\alpha_k - 1}{j_k}
F_{a}^{z}\binom{\{1\}^{n_1}, f_{n_1'}^{0}, \{1\}^{n_2}, f_{n_2'}^{0}, \dots, \{1\}^{n_l}, f_{n_l'}^{0}}{-, \alpha_2 - j_2 + j_3, \dots, \alpha_{n_l'} - j_{n_l'} + j_{n_l'+1}}. \notag
\end{align*}
\end{cor}

To obtain the main theorem, we compute $\tilde{F}$ directly by integration by parts.

\begin{prop}\label{hurie}
The following Fourier series expansion holds: \\
(1)
\begin{align*}
&\tilde{I}_{i\infty}^{z} \binom{\{1\}^{n_1}, f_{n_1'}^{0}, \{1\}^{n_2}, f_{n_2'}^{0}, \dots, \{1\}^{n_l}, f_{n_l'}^{0}}{\alpha_1, \alpha_2, \dots, \alpha_{n_l'}} \notag \\
&= \frac{\Gamma^{(\alpha_1, \dots, \alpha_{n_l'})}}{(-2\pi i)^{\alpha_1 + \dots + \alpha_{n_l'}}} \notag \\
&\quad \times \sum_{m_1, \dots, m_l > 0} \frac{a_{m_1}^{(n_1')} \cdots a_{m_l}^{(n_l')}}{(m_1 + \dots + m_l)^{\alpha_{1,n_1'}} (m_2 + \dots + m_l)^{\alpha_{n_1'+1, n_2'}} \cdots m_l^{\alpha_{n_{l-1}'+1, n_l'}}} q^{m_1 + \dots + m_l}. \notag
\end{align*}
(2)
\begin{align*}
&\tilde{F}_{i\infty}^{z} \binom{\{1\}^{n_1}, f_{n_1'}^{0}, \{1\}^{n_2}, f_{n_2'}^{0}, \dots, \{1\}^{n_l}, f_{n_l'}^{0}}{-, \alpha_2, \dots, \alpha_{n_l'}} \notag \\
&= \frac{\Gamma^{(\alpha_2, \dots, \alpha_{n_l'})}}{(-2\pi i)^{\alpha_2 + \dots + \alpha_{n_l'}}} \notag \\
&\quad \times \sum_{m_1, \dots, m_l > 0} \frac{a_{m_1}^{(n_1')} \cdots a_{m_l}^{(n_l')}}{(m_1 + \dots + m_l)^{\alpha_{2,n_1'}} (m_2 + \dots + m_l)^{\alpha_{n_1'+1, n_2'}} \cdots m_l^{\alpha_{n_{l-1}'+1, n_l'}}} q^{m_1 + \dots + m_l}. \notag
\end{align*}
Here, $\alpha_{n,m} = \alpha_n + \dots + \alpha_m$. In particular, both $\tilde{I}_{i\infty}^{z}$ and $\tilde{F}_{i\infty}^{z}$ are periodic with period $1$.
\end{prop}

\begin{proof}
We prove (1) by induction on the length of the iterated integral. For length $1$,
\begin{align*}
\tilde{I}_{i\infty}^{z} \binom{f_1^0}{\alpha_1}
&= \sum_{m_1 > 0} a_{m_1}^{(1)} \int_{i\infty}^{z} e^{2\pi i m_1 z_1} (z_1 - z)^{\alpha_1 - 1} \, dz_1 \notag \\
&= \sum_{m_1 > 0} a_{m_1}^{(1)} \frac{\Gamma^{(\alpha_1)}}{(-2\pi i m_1)^{\alpha_1}} q^{m_1}, \notag
\end{align*}
which holds. Now assume the claim holds for lengths $\le n$. If $n_1 = 0$,
\begin{align*}
&\tilde{I}_{i\infty}^{z} \binom{f_1^0, \{1\}^{n_2}, f_{n_2'}, \dots, \{1\}^{n_l}, f_{n_l'}^0}{\alpha_1, \dots, \alpha_{n_l'}} \notag \\
&= \int_{i\infty}^{z} f_1^0(z_1)(z_1 - z)^{\alpha_1 - 1} dz_1 \cdot I_{i\infty}^{z_1} \binom{\{1\}^{n_2}, f_{n_2'}, \dots, \{1\}^{n_l}, f_{n_l'}^0}{\alpha_2, \dots, \alpha_{n_l'}} \notag \\
&= \frac{\Gamma^{(\alpha_2, \dots, \alpha_{n_l'})}}{(-2\pi i)^{\alpha_2 + \dots + \alpha_{n_l'}}}
\sum_{m_1, \dots, m_l > 0} \frac{a_{m_1}^{(n_1')} \cdots a_{m_l}^{(n_l')}}{(m_2 + \dots + m_l)^{\alpha_{2,n_2'}} \cdots m_l^{\alpha_{n_{l-1}'+1, n_l'}}} \notag \\
&\quad \times \int_{i\infty}^{z} e^{2\pi i (m_1 + \dots + m_l) z_1} (z_1 - z)^{\alpha_1 - 1} \, dz_1 \notag \\
&= \frac{\Gamma^{(\alpha_1, \dots, \alpha_{n_l'})}}{(-2\pi i)^{\alpha_1 + \dots + \alpha_{n_l'}}}
\sum_{m_1, \dots, m_l > 0} \frac{a_{m_1}^{(n_1')} \cdots a_{m_l}^{(n_l')}}{(m_1 + \dots + m_l)^{\alpha_{1,n_1'}} \cdots m_l^{\alpha_{n_{l-1}'+1, n_l'}}} q^{m_1 + \dots + m_l}, \notag
\end{align*}
which proves the case $n_1 = 0$. For $n_1 \ne 0$,
\begin{align*}
&\tilde{I}_{i\infty}^{z} \binom{\{1\}^{n_1}, f_{n_1'}^{0}, \{1\}^{n_2}, f_{n_2'}^{0}, \dots, \{1\}^{n_l}, f_{n_l'}^{0}}{\alpha_1, \alpha_2, \dots, \alpha_{n_l'}} \notag \\
&= \int_{i\infty}^{z} (z_1 - z)^{\alpha_1 - 1} I_{i\infty}^{z_1} \binom{\{1\}^{n_1 - 1}, f_{n_1'}, \dots, \{1\}^{n_l}, f_{n_l'}^0}{\alpha_2, \dots, \alpha_{n_l'}} \notag \\
&= \frac{\Gamma^{(\alpha_2, \dots, \alpha_{n_l'})}}{(-2\pi i)^{\alpha_2 + \dots + \alpha_{n_l'}}}
\sum_{m_1, \dots, m_l > 0} \frac{a_{m_1}^{(n_1')} \cdots a_{m_l}^{(n_l')}}{(m_1 + \dots + m_l)^{\alpha_{2,n_1'}} \cdots m_l^{\alpha_{n_{l-1}'+1, n_l'}}} \notag \\
&\quad \times \int_{i\infty}^{z} e^{2\pi i (m_1 + \dots + m_l) z_1} (z_1 - z)^{\alpha_1 - 1} \, dz_1 \notag \\
&= \frac{\Gamma^{(\alpha_1, \dots, \alpha_{n_l'})}}{(-2\pi i)^{\alpha_1 + \dots + \alpha_{n_l'}}}
\sum_{m_1, \dots, m_l > 0} \frac{a_{m_1}^{(n_1')} \cdots a_{m_l}^{(n_l')}}{(m_1 + \dots + m_l)^{\alpha_{1,n_1'}} \cdots m_l^{\alpha_{n_{l-1}'+1, n_l'}}} q^{m_1 + \dots + m_l}, \notag
\end{align*}
which completes the proof of (1). Statement (2) follows from (1).
\end{proof}

\begin{prop}\label{prop3}
The following Mellin transforms hold:
\begin{align*}
&\int_{i\infty}^{0} \tilde{I}_{i\infty}^{z} \binom{\{1\}^{n_1}, f_{n_1'}^{0}, \{1\}^{n_2}, f_{n_2'}^{0}, \dots, \{1\}^{n_l}, f_{n_l'}^{0}}{\alpha_1, \alpha_2, \dots, \alpha_{n_l'}} z^s \frac{dz}{z} \notag \\
&= \Gamma^{(s, \alpha_1, \dots, \alpha_{n_l'})} \cdot L\binom{f_{n_1'}, \dots, f_{n_l'}}{s + \alpha_{1,n_1'}, \alpha_{n_1'+1,n_2'}, \dots, \alpha_{n_{l-1}'+1,n_l'}}, \notag
\end{align*}
\begin{align*}
&\int_{i\infty}^{0} \tilde{F}_{i\infty}^{z} \binom{\{1\}^{n_1}, f_{n_1'}^{0}, \{1\}^{n_2}, f_{n_2'}^{0}, \dots, \{1\}^{n_l}, f_{n_l'}^{0}}{-, \alpha_2, \dots, \alpha_{n_l'}} z^s \frac{dz}{z} \notag \\
&= \Gamma^{(s, \alpha_2, \dots, \alpha_{n_l'})} \cdot L\binom{f_{n_1'}, \dots, f_{n_l'}}{s + \alpha_{2,n_1'}, \alpha_{n_1'+1,n_2'}, \dots, \alpha_{n_{l-1}'+1,n_l'}}. \notag
\end{align*}
\end{prop}

\begin{proof}
This follows immediately from Proposition~\ref{hurie}.
\end{proof}

Using these propositions, we now prove Theorem~\ref{thI}.

\begin{proof}
By applying Proposition~\ref{prop1}, Corollary~\ref{col2}, and Proposition~\ref{prop3}, we compute as follows:
\begin{align*}
I_{i\infty}^{0} &\binom{f_1, \dots, f_n}{s, \alpha_2, \dots, \alpha_n} \notag \\
&= \sum_{l=1}^{n} \sum_{\substack{n_{l+1}' = n+1 \\ n_1, \dots, n_{l+1} \ge 0}} A_{n_1, \dots, n_l} B_{n_{l+1}}(s, \alpha_2, \dots, \alpha_n) \notag \\
&\quad \times \int_{i\infty}^0F_{i\infty}^{z} \binom{\{1\}^{n_1}, f_{n_1'}^0, \dots, \{1\}^{n_l}, f_{n_l'}^0}{-, \alpha_2, \dots, \alpha_{n_l'-1}, \alpha_{n_l', n}} z^{s-1} dz \notag \\
&= \sum_{l=1}^{n} \sum_{\substack{n_{l+1}' = n+1 \\ n_1, \dots, n_{l+1} \ge 0}}^{\hspace{6mm}'} A_{n_1, \dots, n_l} B_{n_{l+1}}(s, \alpha_2, \dots, \alpha_n) \notag \\
&\quad \times \sum_{\substack{0 \le j_r < \alpha_r + j_{r+1} \\ (2 \le r \le n_l')}} \binom{\alpha_{n_l',n}-1}{j_{n_l}'}\prod_{k=2}^{n_l'-1} \binom{\alpha_k + j_{k+1} - 1}{j_k} \notag \\
&\quad \times \int_{i\infty}^0\tilde{F}_{i\infty}^{z} \binom{\{1\}^{n_1}, f_{n_1'}^0, \dots, \{1\}^{n_l}, f_{n_l'}^0}{-, \alpha_2 - j_2 + j_3, \dots, \alpha_{n_l'-1} - j_{n_l'-1} + j_{n_l}, \alpha_{n_l', n} - j_{n_l'}} z^{s + j_2 - 1} \notag \\
&= \sum_{l=1}^{n} \sum_{\substack{n_{l+1}' = n+1 \\ n_1, \dots, n_{l+1} \ge 0}}^{\hspace{6mm}'} A_{n_1, \dots, n_l} B_{n_{l+1}}(s, \alpha_2, \dots, \alpha_n) \notag \\
&\quad \times \sum_{\substack{0 \le j_r < \alpha_r + j_{r+1} \\ (2 \le r \le n_l')}} \binom{\alpha_{n_l',n}-1}{j_{n_l}'}\prod_{k=2}^{n_l'-1} \binom{\alpha_k + j_{k+1} - 1}{j_k} \Gamma^{(s + j_2, \alpha_2 - j_2 + j_3, \dots, \alpha_{n_l'-1} - j_{n_l'-1} + j_{n_l'}, \alpha_{n_l', n} - j_{n_l'})} \notag \\
&\quad \times L \binom{f_{n_1'}, \dots, f_{n_l'}}{s + \alpha_{2,n_1'} + j_{n_1'+1}, \boldsymbol{\alpha}_{n_1, \dots, n_l}(\mathbb{J}), \alpha_{n_{l-1}'+1,n} - j_{n_{l-1}'+1}}. \notag
\end{align*}
Finally, we use the identity
\begin{align*}
&\binom{\alpha_{n_l',n}-1}{j_{n_l}'}\prod_{k=2}^{n_l'-1} \binom{\alpha_k + j_{k+1} - 1}{j_k} \Gamma^{(s + j_2, \alpha_2 - j_2 + j_3, \dots, \alpha_{n_l'-1} - j_{n_l'-1} + j_{n_l'}, \alpha_{n_l', n} - j_{n_l'})} \notag \\
&= \Gamma^{(s, \alpha_2, \dots, \alpha_{n_l'-1}, \alpha_{n_l', n})} \binom{s + j_2 - 1}{j_2} \prod_{k=3}^{n_l'} \binom{\alpha_{k-1} + j_k - 1}{j_k}, \notag
\end{align*}
to complete the proof.
\end{proof}

 We now proceed to prove Theorem~\ref{thS}.
\begin{proof}
We prove the theorem by induction on $n$.  
For $n = 1$ is classical. 
Assume the statement holds for all natural numbers less than or equal to $n$, and consider the case $n > 1$.

By Proposition~\ref{prop3}, we have:
\begin{align*}
L\binom{f_1, \dots, f_n}{s, \alpha_2, \dots, \alpha_n}
&= \frac{1}{\Gamma^{(s, \alpha_2, \dots, \alpha_n)}}
\sum_{\substack{0 \le j_r < \alpha_r \\ (2 \le r < n)}}
\prod_{k=2}^{n} (-1)^{j_k} \binom{\alpha_k}{j_k} \notag \\
&\quad \times I_{i\infty}^{0}
\binom{f_1^0, \dots, f_n^0}{s + j_2, \alpha_2 - j_2 + j_3, \dots, \alpha_n - j_n + j_{n+1}}, \notag
\end{align*}
where $f^0 = f - a_0$. The statement reduces to a classical result.
The final integral can be expressed as a linear combination of iterated integrals of the form
\[
I_{i\infty}^0 \binom{\{1\}^{n_1}, f_{n_1'}, \dots, \{1\}^{n_l}, f_{n_l'}, \{1\}^{n_{l+1}}}
{s + j_2, \alpha_2 - j_2 + j_3, \dots, \alpha_n - j_n + j_{n+1}},
\]
where $1 \le l \le n$ and $n_{l+1}' = n + 1$.

When $l = n$, this becomes the original integral
\[
I_{i\infty}^{0} \binom{f_1, \dots, f_n}{s + j_2, \alpha_2 - j_2 + j_3, \dots, \alpha_n - j_n + j_{n+1}}.
\]
When $l < n$, it can be expressed as a linear combination of multiple $L$-functions of length $l$.
Therefore, by the induction hypothesis, the claim follows.
\end{proof}

\section{Special values of modular iterated integral}

Multiple zeta values can be expressed as iterated integrals on \(\mathbb{P}^1 \setminus \{0,1,\infty\}\).  
Moreover, there exists an isomorphism
\[
\lambda: Y_0(4) \xrightarrow{\sim} \mathbb{P}^1 \setminus \{0,1,\infty\}, \quad \lambda = \frac{16F}{G}.
\]
Using this isomorphism, we consider how multiple zeta values can be written as modular iterated integrals by pulling them back to \(Y_0(4)\).  
To compute the pullbacks of \(\omega_0 = \frac{dz}{z}\) and \(\omega_1 = \frac{dz}{1 - z}\), it is convenient to use the \(\eta\)-function, which we therefore define.

We define the following elements \( F, G \in M_2(\Gamma_0(4)) \):
\begin{align*}
F(z) &= -\frac{1}{24} \left( E_2(z) - 3E_2(2z) + 2E_2(4z) \right) 
= \sum_{\substack{n > 0 \\ n \text{ odd}}} \sigma_1(n) q^n, \\
G(z) &= \theta(z)^4,
\end{align*}
where 
\[
E_k(z)= 1 - \frac{2k}{B_k} \sum_{n=1}^\infty \sigma_{k-1}(n) q^n,
\qquad
\theta(z) = \sum_{n \in \mathbb{Z}} e^{2\pi i z n^2} = \sum_{n \in \mathbb{Z}} q^{n^2}.
\]
Moreover, \( F \) and \( G \) form a basis for \( M_2(\Gamma_0(4)) \).

\begin{dfn}
For a point \( z \) in the upper half-plane, the Dedekind \(\eta\)-function is defined by
\begin{align*}
\eta(z) = q^{\frac{1}{24}} \prod_{n=1}^{\infty} (1 - q^n),
\end{align*}
where \( q = e^{2\pi i z} \).
\end{dfn}

It is known that the Dedekind \(\eta\)-function satisfies the following proposition.

\begin{prop} \label{eta}
For any integer \( l \), the following identity holds:
\begin{align*}
\frac{d}{dz} \log(\eta(lz)) = \frac{2\pi i l}{24} E_2(lz),
\end{align*}
where \( E_2(z) \) denotes the Eisenstein series of weight 2.
\end{prop}

\begin{prop}
For the functions \( F, G \) , the following identities hold:
\begin{align*}
F(z) &= \frac{\eta(4z)^8}{\eta(2z)^4}, \\
G(z) &= \frac{\eta(2z)^{20}}{\eta(z)^8 \eta(4z)^8}, \\
G(z) - 16F(z) &= \frac{\eta(z)^8}{\eta(2z)^4}.
\end{align*}
\end{prop}

We now compute the pullback of \(\omega_0 = \frac{dz}{z}\) via \(\lambda\), using \Cref{eta}:
\begin{align*}
\frac{d\lambda(z)}{\lambda(z)} 
&= d\log(\lambda(z)) \notag \\
&= d\log\left(16 \cdot \frac{F(z)}{G(z)}\right) \notag \\
&= d\log\left(\frac{\eta(4z)^{16} \eta(z)^8}{\eta(2z)^{24}}\right) \notag \\
&= 16 \, d\log \eta(4z) + 8 \, d\log \eta(z) - 24 \, d\log \eta(2z) \notag \\
&= \frac{2\pi i}{3} \left( E_2(z) + 8 E_2(4z) - 6 E_2(2z) \right) dz \notag \\
&= 2\pi i \left( G(z) - 16 F(z) \right) dz. 
\end{align*}
The final equality follows by comparing the coefficients in the \(q\)-expansions.  
Similarly, we compute the pullback of \(\omega_1 = \frac{dz}{1 - z}\) via \(\lambda\), again using \Cref{eta}:

\begin{align*}
\frac{d\lambda(z)}{1 - \lambda(z)} 
&= -d\log(1 - \lambda(z)) \notag \\
&= -d\log\left(1 - \frac{16F(z)}{G(z)}\right) \notag \\
&= -d\log\left(\frac{G(z) - 16F(z)}{G(z)}\right) \notag \\
&= -d\log\left(\frac{\eta(z)^{16} \eta(4z)^8}{\eta(2z)^{24}}\right) \notag \\
&= -\frac{4\pi i}{3} \left( E_2(z) + 2 E_2(4z) - 3 E_2(2z) \right) dz \notag \\
&= 32 \pi i F(z) dz. 
\end{align*}

We find that multiple zeta values can be expressed as iterated integrals on \(Y_0(4)\) in terms of the above pullbacks.

\begin{prop}
Let \( \zeta(k_1, \dots, k_d) \) be a multiple zeta value of weight \( w = k_1 + \dots + k_d \). Then the following identity holds:
\begin{align*}
\zeta(k_1, \dots, k_d)
&= \int_{\mathrm{dch}} \omega_1 \omega_0^{k_1 - 1} \cdots \omega_1 \omega_0^{k_d - 1} \notag \\
&= \left(2\pi i \right)^w 16^d \int_{\lambda^*(\mathrm{dch})}
\Big( F(z)\,dz \cdot \big( (G(z) - 16F(z))\,dz \big)^{k_1 - 1} \notag \\
&\hspace{6em} \cdots
F(z)\,dz \cdot \big( (G(z) - 16F(z))\,dz \big)^{k_d - 1} \Big).
\end{align*}
\end{prop}

\end{document}